\documentclass[12pt,a4paper]{article}
\usepackage{amsthm}
\usepackage{amsmath}
\usepackage{epsfig}
\usepackage{color} 
\usepackage{amssymb}

\newtheorem{theorem}{Theorem}[section]
\newtheorem{proposition}[theorem]{Proposition}
\newtheorem{lemma}[theorem]{Lemma}
\newtheorem{corollary}[theorem]{Corollary}

\newtheorem*{maintheorem}{Main Theorem}

\newtheorem*{mainconjecture}{Conjecture}

\newcommand{\inv}{^{-1}}

\begin{document}

\title{\large{\textbf{NOWHERE-ZERO FLOWS ON\\ SIGNED EULERIAN GRAPHS}}}

\author{
Edita M\'a\v cajov\' a
and Martin \v{S}koviera\\[3mm]
Department of Computer Science\\
Faculty of Mathematics, Physics and Informatics\\
Comenius University\\
842 48 Bratislava, Slovakia\\[2mm]
{\small\tt macajova@dcs.fmph.uniba.sk}\\[-1mm]
{\small\tt skoviera@dcs.fmph.uniba.sk}}

\date{}

\maketitle

\begin{abstract}
This paper is devoted to a detailed study of nowhere-zero flows
on signed eulerian graphs. We generalise the well-known fact
about the existence of nowhere-zero $2$-flows in eulerian
graphs by proving that every signed eulerian graph that admits
an integer nowhere-zero flow has a nowhere-zero $4$-flow. We
also characterise signed eulerian graphs with flow number $2$,
$3$, and $4$, as well as those that do not have an integer
nowhere-zero flow. Finally, we discuss the existence of
nowhere-zero $A$-flows on signed eulerian graphs for an
arbitrary abelian group~$A$.

\vskip8pt\noindent{\bf Keywords:} Eulerian graph;  signed
graph; bidirected graph; nowhere-zero flow
\end{abstract}

\section{Introduction}
A nowhere-zero flow is an assignment of an orientation and a
nonzero value from an abelian group $A$ to each edge of a graph
in such a way that the Kirchhoff current law is fulfilled at
each vertex: the sum of in-flowing values equals the sum of
out-flowing values. This concept was introduced by Tutte
\cite{Tutte1} in 1949 and has been extensively studied by many
authors. There is an analogous concept of a nowhere-zero flow
that uses bidirected edges instead of directed ones, first
systematically developed by Bouchet \cite{Bouchet} in 1983. A
bidirected edge is an edge consisting of two half-edges which
receive separate orientations; a bidirected graph is one where
each edge has been bidirected. A nowhere-zero flow on a
bidirected graph is formed by valuating each edge with a
nonzero element of $A$ in such a way that, for every
vertex~$v$, the sum of values on the half-edges directed to $v$
equals the sum of values on the half-edges directed out of $v$.
If the half-edges of each edge of $G$ are aligned, $G$
essentially becomes a directed graph, which implies that the
bidirected variant of a flow is more general than the simply
directed one.

It is obvious that the choice of an orientation in the simply
directed case is immaterial: it is always possible to reverse
an edge and change its flow value to the opposite value. The
same is true in the bidirected case as long as both half-edges
of an edge are reversed at once, leaving the partition into
aligned and non-aligned edges unaltered. It is somewhat less
obvious that the flow on a bidirected graph can be preserved
even if this edge-partition is disturbed by reversing all the
half-edges around a vertex. If we keep the flow values intact,
the Kirchhoff law will not be violated. This indicates that the
invariance of flows on bidirected graphs is best understood in
terms of signed graphs where the aligned edges are positive,
non-aligned edges are negative, and the just described
operation is the familiar vertex-switching from signed graph
theory. To summarise, the existence of a flow is a property of
a signed graph invariant under switching equivalence and
independent of its particular bidirection. An easy switching
argument further shows that the simply directed case
corresponds to the case of balanced signed graphs, those where
every circuit has an even number of negative edges.

The study of flows on signed (or, equivalently, bidirected)
graphs seems to have its roots in the study of embeddings of
graphs in nonorientable surfaces. Signed graphs provide a
convenient language for description of such embeddings and
bidirected graphs arise naturally as duals of directed graphs.
In 1968, Youngs \cite{Youngs1, Youngs2} employed nowhere-zero
flows on signed cubic graphs with values in cyclic groups,
combined with surface duality, to construct nonorientable
triangular embeddings of certain complete graphs. A duality
relationship between local tensions and flows on graphs
embedded in nonorientable surfaces was the main motivation for
Bouchet's work \cite{Bouchet} on integer flows in signed
graphs. As in the unsigned case, the fundamental problem here
is to find, for a given signed graph, a nowhere-zero flow with
the smallest maximum absolute edge-value. If this value is
$k-1$, we speak of a nowhere-zero $k$-flow. In the same paper
Bouchet showed that every graph that admits a nowhere-zero
integer flow has a nowhere-zero $216$-flow. He also proposed
the following conjecture that has become an incentive for much
of the current research in the area.

\begin{mainconjecture} {\rm (Bouchet's $6$-Flow Conjecture)}
Every signed graph with a nowhere-zero integer flow has a
nowhere-zero $6$-flow.
\end{mainconjecture}

The present status of this conjecture can be summarised as
follows. Bouchet's $216$-flow theorem was improved in 1987 by
Z\'yka \cite{Zyka} to a nowhere-zero 30-flow, and very recently
by DeVos \cite{Devos} to a nowhere-zero $12$-flow, which is
currently the best general approximation of Bouchet's
conjecture. In 2005, the existence of a nowhere-zero $6$-flow
was established by Xu and Zhang \cite{Xu} for every
$6$-edge-connected signed graph with a nowhere-zero integer
flow. In 2011, Raspaud and Zhu \cite{Raspaud} proved that every
$4$-edge-connected signed graph admitting an integer
nowhere-zero flow has a nowhere-zero $4$-flow, which is best
possible. Recently, Wu et al. \cite{WuYZZ} proved that
every $8$-edge-connected signed graph admitting an integer
nowhere-zero flow has a nowhere-zero $3$-flow.

In spite of these results, very little is known about exact
flow numbers for various classes of signed graphs.
Surprisingly, the situation is open even for signed eulerian
graphs, although Xu and Zhang \cite[Proposition 1.4]{Xu} proved
that a signed eulerian graph with an even number of negative
edges has a nowhere-zero $2$-flow. The purpose of the present
paper is to strengthen this result by proving the following
theorem which provides a classification of signed eulerian graphs
according to their flow numbers. The smallest example in each of 
the four resulting classes is displayed in Figure~\ref{fig:prototypes}.

\begin{maintheorem} 
Let $G$ be a signed eulerian graph and let $\Phi(G)$ denote the flow 
number of $G$. Then
\begin{itemize}
\item[{\rm (a)}] $G$ has no nowhere-zero flow if and only
    if $G$ is unbalanced and $G-e$ is balanced for some
    edge $e$;

\item[{\rm (b)}] $\Phi(G)=2$ if and only if $G$ has an even
    number of negative edges;

\item[{\rm (c)}] $\Phi(G)=3$ if and only if $G$ can be
    decomposed into three eulerian subgraphs, with an odd
    number of negative edges each, that share a common
    vertex;

\item[{\rm (d)}] $\Phi(G)=4$ otherwise.
\end{itemize}
\end{maintheorem}

\begin{figure}[htbp]\label{fig:prototypes}
  \centerline{
     \scalebox{0.6}{
       \input{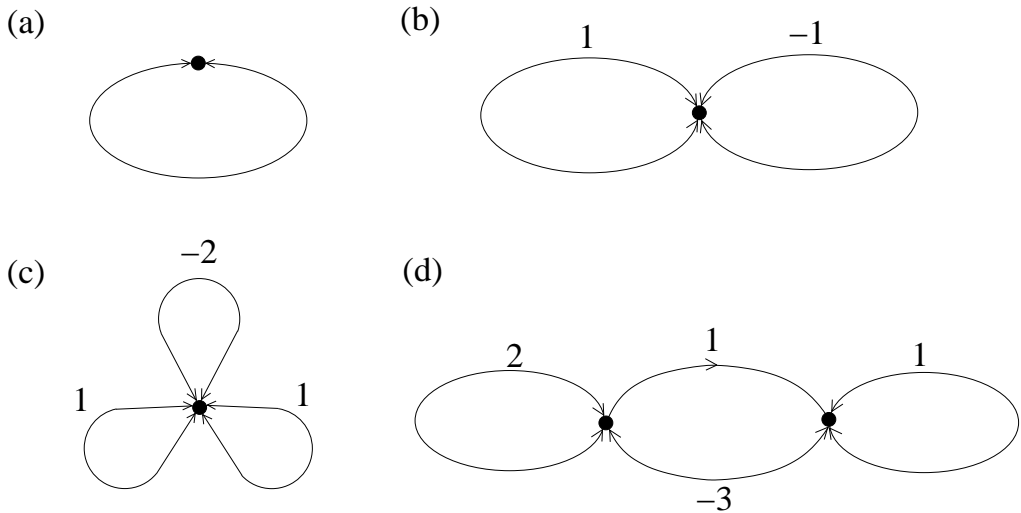}
     }
  }
\caption{Examples of signed eulerian graphs from Main Theorem}
\end{figure}

Our paper is organised as follows. The next section gives a
brief introduction to signed graphs collecting the concepts and
results to be used later in this paper; for more information on
signed graphs we refer the reader to \cite{Zaslav1, Zaslav2,
Zaslav3}. The third section concentrates on the problem of
the existence of nowhere-zero integer flows in signed graphs and
provides a characterisation of graphs that admit a nowhere-zero
integer flow. A corollary of this characterisation states that
a signed eulerian graph has no nowhere-zero integer flow if and
only if its signature is switching-equivalent to one with a
single negative edge. In Section~$4$ we prove that all other
eulerian graphs admit a nowhere-zero $4$-flow and in Section~5
we characterise those with flow number~$3$. Finally, the last
section deals with the existence of nowhere-zero $A$-flows on
signed eulerian graphs for an arbitrary abelian group $A$.

\medskip

We conclude this section with a few terminological and
notational remarks. All our graphs are finite and may have
multiple edges and loops. A graph is called \textit{eulerian}
if it is connected, with vertices of even valency. A
\textit{circuit} is a connected $2$-regular graph, and a
\textit{cycle} is a graph that has a decomposition (possibly
empty) into edge-disjoint circuits. A set of edges is often
identified with the subgraph it induces; this should not cause
any confusion. Each walk is understood to be directed from its
initial to its terminal vertex. The walk obtained from $W$ by
reversing the direction will be denoted by $W\inv$. If $x$ and
$y$ are two vertices of $W$ listed in the order of their
appearance on~$W$, we let $W[x,y]$ denote the portion of $W$
initiating at $x$ and terminating at $y$. Finally, if $W_1$ is a 
$u$-$v$-walk and $W_2$ is a $v$-$w$-walk, then $W_1W_2$ denotes the
$u$-$w$-walk where the traversal of $W_1$ is followed by the
traversal of $W_2$.

\section{Fundamentals of signed graphs}

A \textit{signed graph} is a graph $G$ endowed with a
\textit{signature}, a mapping that assigns $+1$ or $-1$ to each
edge. In our notation, the signature is usually implicit in the
notation of the graph itself; only when needed, it will be
denoted by $\sigma_G$, or simply by $\sigma$, if $G$ is clear
from the context. As most graphs considered here will be
signed, the term \textit{graph} will usually be used to mean a
signed graph, and the adjective \textit{signed} will be added
only for emphasis. An unsigned graph will be regarded as a
signed graph having the \textit{all-positive signature}
$\sigma_G\equiv +1$.

The actual distribution of edge signs in a signed graph is not
very important. What is fundamental is the product of signs on
each of its circuits. This constitutes the concept of a
\textit{balance} in a signed graph: Let $F$ be a subgraph or a
set of edges of a signed graph~$G$. We define the \textit{sign}
of $F$, denoted by $\sigma_G(F)$, as the product of the signs
of all edges in~$F$. Thus, every subgraph or set of edges can
either be \textit{positive} or \textit{negative}, depending on
whether its sign is $+1$ or $-1$. A closed walk, in particular
a circuit, is said to be \textit{balanced} if its sign is $+1$;
otherwise it is \textit{unbalanced}. A graph is
\textit{balanced} if it contains no unbalanced circuit. The
collection of all balanced circuits of a signed graph is its
most fundamental characteristic: signed graphs that have the
same underlying graphs and the same sets of balanced circuits
are considered to be \textit{identical}, irrespectively of
their actual signatures. Signatures of identical signed graphs
are called \textit{equivalent}.

Let $G$ be a signed graph and let $v$ be a vertex of $G$. It is
clear that if we interchange the signs of all non-loop edges
incident with $v$, the set of balanced circuits will not
change. This operation, called \textit{switching} at $v$, thus
produces an equivalent signature. More generally, for a set $U$
of vertices of $G$ we define \textit{switching} at $U$ as the
interchange of signs on all edges with exactly one end-vertex
in $U$. It is easy to see that switching the signature at $U$
has the same effect as switching at all the vertices of $U$ in
a succession. Furthermore, it is not difficult to prove that
two signatures are equivalent if and only if they are
\textit{switching-equivalent}, that is, if they can be turned
into each other by a sequence of vertex switchings
\cite[Proposition 3.2]{Zaslav2}. In particular, a signed graph
is balanced if and only if its signature is
switching-equivalent to the all-positive signature.

The following characterisation of balanced graphs due to Harary
\cite{Harary} has the same spirit as the well-known
characterisation of bipartite graphs.

\begin{theorem}\label{thm:bal-bipart}
\emph{(Harary's Balance Theorem \cite{Harary})} A signed graph
is balanced if and only if its vertex set can be partitioned
into two sets (either of which may be empty) in such a way that
each edge between the sets is negative and each edge within a
set is positive.
\end{theorem}

The partition of the vertex set of a signed into two sets
mentioned in the previous theorem will be called a
\textit{balanced bipartition}. It is useful to realise that the
balanced bipartition of a signed graph depends on the chosen
signature. However, once a signature is fixed and the graph is
connected, the balanced bipartition is uniquely determined.

A signed graph $G$ is said to be \textit{antibalanced} if
replacing its signature $\sigma_G$ with $-\sigma_G$ makes it
balanced. This definition immediately implies that a signed
graph is antibalanced if and only if its signature is
equivalent to the all-negative signature. Hence, a circuit of
an antibalanced graph is balanced precisely when it has an even
length. Consequently, an antibalanced graph is balanced if and
only if its underlying unsigned graph is bipartite.

The following is a direct consequence of Harary's Balance Theorem.

\begin{corollary}\label{col:antibal-bipart}
A signed graph is antibalanced if and only if its vertex set
can be partitioned into two sets (either of which may be empty)
in such a way that each edge between the sets is positive and
each edge within a set is negative.
\end{corollary}
The partition of the vertex set of a signed into two sets
mentioned in the previous corollary will be called an
\textit{antibalanced bipartition}.

\section{Flows and flow-admissible signed graphs}

The aim of this section is to introduce the concepts related to
flows in signed graphs and to characterise signed graphs that
admit a nowhere-zero integer flow.

As in the case of unsigned graphs, the definition of a flow
calls for an orientation of the underlying  graph. The signed
version, however, requires a bidirection of edges rather than
just a usual orientation. For this purpose we regard each edge
$e$, including loops, as consisting of two half-edges, each
half-edge being incident with only one end-vertex of $e$. An
edge is \textit{bidirected} if each of its two half-edges is
individually directed from or to its associated end-vertex.
Thus every edge has four possible orientations which fall into
two types:
\begin{itemize}
\item Two of the four orientations have one half-edge
    directed from and the other half-edge towards its
    end-vertex. Such bidirected edges will be identified
    with the usual directed edges and called
    \textit{ordinary} edges.

\item The other type of bidirection has either both
    half-edges directed from or both half-edges oriented
    towards their end-vertices. In the former case, the
    edge is said to be \textit{extroverted}, and in the
    latter case it is said to be \textit{introverted}.
    Extroverted and introverted edges are collectively
    called \textit{broken} edges.
\end{itemize}

A bidirected edge $e$ incident with a vertex $v$ is said to be
directed \textit{out} of $v$ if its half-edge incident with $v$
is directed out of $v$. Similarly, $e$ is said to be directed
\textit{to} $v$ if its half-edge incident with $v$ is directed
to $v$. For an arbitrary vertex $v$, the set of all edges
directed out of $v$ is denoted by $E^{\text{out}}(v)$, and the
set of all edges directed to $v$ is denoted by
$E^{\text{in}}(v)$; the set of all edges incident with $v$ is
denoted simply by $E(v)$.

Every bidirected edge $e$ has a well-defined \textit{reverse}
$e\inv$ obtained by reversing the orientation of both constituting
half-edges. In particular, if $e$ is an ordinary edge, then
$e\inv$ is also ordinary but with the reversed order of its
end-vertices; if $e$ is extroverted, then~$e\inv$ is introverted,
and vice versa.

Given a signed graph $G$, an \textit{orientation} of $G$ is an
assignment of a bidirection to each edge in such a way that the
following compatibility rule is fulfilled: every positive edge
receives an orientation that turns it into an ordinary edge,
while every negative edge receives an orientation that turns it
into a broken edge. Thus, endowing $G$ with an orientation
makes $G$ a \textit{bidirected} graph. Conversely, every
bidirected graph can be regarded as a signed graph in which
ordinary edges are positive and broken edges are negative. In
other words, a bidirected graph and a signed graph endowed with
an orientation are equivalent concepts.

If a signed graph $G$ is bidirected, switching its signature at
an arbitrary vertex or set of vertices must cause the change of
its orientation so that the compatibility rule remains
fulfilled. Therefore we define \textit{switching} at an
arbitrary set $U$ of vertices as an operation that reverses the
orientation of each half-edge incident with a vertex in $U$ and
consequently changes the sign of each edge with exactly one end
in $U$.

We now proceed to the definition of a flow on a signed graph.
Let $G$ be a signed graph which has been endowed with an
arbitrary compatible orientation. A mapping $\xi\colon E(G)\to
A$ with values in an abelian group $A$ is called an
\textit{$A$-flow} on $G$ provided that the following continuity
condition, the \textit{Kirchhoff law}, is satisfied at each
vertex $v$ of $G$:
\begin{eqnarray*}
\sum_{e\in E^{\text{out}}(v)}\xi(e)\ \ -\sum_{e\in
E^{\text{in}}(v)}\xi(e)=0.
\end{eqnarray*}
An $A$-flow $\xi$ is said to be \textit{nowhere-zero} if
$\xi(e)\ne 0$ for each edge $e$ of $G$. A \textit{nowhere-zero}
$k$-flow is a $\mathbb{Z}$-flow that takes its values from the
set $\{\pm 1,\ldots \pm(k - 1)\}$. Clearly, a signed graph that
has a nowhere-zero $k$-flow also has a nowhere-zero $(k
+1)$-flow. The smallest integer $k$ for which a signed graph
$G$ has a nowhere-zero $k$-flow is called the \textit{flow
number} of $G$ and is denoted by $\Phi(G)$.

Take an arbitrary flow $\xi$ on a signed graph $G$. It is
obvious that if we reverse the orientation of an arbitrary edge
$e$ of $G$ -- that is, if we replace $e$ with $e\inv$ -- and
set $\xi(e\inv)=-\xi(e)$, Kirchhoff's law remains satisfied.
Similarly, if we switch the signature at some vertex $v$ of
$G$, we change the bidirection of each edge $e$ incident with
$v$ to a well defined new bidirection $e'$ whose type differs
from that of $e$. At the same time we interchange the roles of
the sets $E^{\text{out}}(v)$ and $E^{\text{in}}(v)$. It follows
that if we set $\xi(e')=\xi(e)$, Kirchhoff's law will continue
to hold at each vertex of~$G$. The conclusion is that any flow
on a signed graph $G$ is essentially independent of the chosen
orientation and the particular signature and only depends on
the switching class of the signed graph itself. This makes the
definition of a flow on a signed graph completely analogous to
the definition of a flow on an unsigned graph.

In spite of this apparent similarity of definitions, flows on
signed graphs can sometimes have a rather unexpected behaviour.
For example, the following useful property has no analogue in
unsigned graphs.

\begin{lemma}\label{lemma:extro}
For any flow on a signed graph, the sum of values on negative
edges taken in the extroverted orientation is zero.
\end{lemma}

\begin{proof}
Take a flow on a signed graph $G$, redirect each negative edge
of $G$ to make it extroverted, and change the flow values
accordingly. By the Kirchhoff law, the total outflow from every
vertex is zero, so the sum of all outflows is again zero. Each
edge contributes to this sum twice, a positive edge with
opposite signs while a negative edge with the same sign.
Ignoring the values on positive edges, which sum to zero, we
infer that the doubled sum of values on the negative edges
equals zero. The result follows.
\end{proof}

It is well known that an unsigned graph admits a nowhere-zero
flow if and only if it is bridgeless, irrespectively of the
group employed. In contrast, the existence of a nowhere-zero
flow on a signed graph may depend on the chosen group, and
exceptional graphs are less straightforward to describe. In the
present section we focus on graphs that admit a nowhere-zero
integer flow and call them \textit{flow-admissible}. Other
groups will be discussed in Section~6.

Consider a signed graph $G$. If $G$ is balanced, then a simple
switching argument shows that each flow on $G$ corresponds to a
flow on the corresponding unsigned graph. Therefore a balanced
signed graph is flow-admissible if and only if it is
bridgeless. In contrast, an unbalanced bridgeless signed graph
may fail to be flow-admissible while a graph having a bridge
may happen to be flow admissible. For example, Bouchet in
\cite[Lemma 2.4]{Bouchet} observed that a $2$-edge-connected
graph with a single negative edge is never flow-admissible.

Our next aim is to characterise unbalanced flow-admissible
signed graphs. Before proceeding to the result, some
preparations are in order.

We describe a simple technique which is useful for
construction of nowhere-zero flows on signed graphs. Consider a
pair of adjacent edges $e$ and $f$ sharing a vertex $v$ in a
bidirected signed graph. We say that the walk $ef$ is
\textit{consistently directed} at $v$ if exactly one of the two
half-edges incident with $v$ is directed to~$v$. A trail $P$ is
said to be \textit{consistently directed} if all pairs of
consecutive edges of $P$ are consistently directed. Now let $G$
be a signed graph carrying an $A$-flow $\phi$, and let $P$ be a
$u$-$v$-trail in $G$; if $\phi$ has not been specified, we are
assuming that $\phi=0$. By \emph{sending a value $b\in A$ from
$u$ to $v$ along $P$} we mean the modification of $\phi$ into a
new valuation $\phi'\colon E(G)\to A$ defined as follows. We
keep the values of $\phi$ everywhere except on the edges of
$P$. On $P$ we change the orientation of edges in such a way
that the initial edge is directed from $u$, and $P$ is
consistently directed at each internal vertex; we change the
flow values accordingly. Finally, for each edge $e$ on $P$ we
replace the current value $\phi(e)$ with $\phi(e)+b$.

Under the new valuation $\phi'$, Kirchhoff's law will be
satisfied at each inner vertex of~$P$. Moreover, if $P$ is
closed -- that is if $u=v$ -- and the number of negative edges
on $P$ is even, Kirchhoff's law will be satisfied at $u$ as
well. Thus sending any value along a closed trail with an even
number of negative edges will turn an $A$-flow into another
$A$-flow.

Define a \textit{signed circuit} as a signed graph of any of
the following three types:
\begin{itemize}
\item[(1)] a balanced circuit,
\item[(2)] the union of two unbalanced circuits that meet
    at a single vertex, or
\item[(3)] the union of two disjoint unbalanced circuits
    with a path that meets the circuits only at its ends,
\end{itemize}
A signed circuit falling under item (2) or (3) will be called
an \textit{unbalanced bicircuit}.

Observe that every signed circuit admits a nowhere-zero
integer-flow. Indeed, a signed circuit $S$ satisfying (1) or
(2) constitutes a closed trail with an even number of negative
edges, so sending value $1$ along the trail produces a
nowhere-zero $2$-flow on $S$. If $S$ is an unbalanced bicircuit
satisfying (3) and consisting of unbalanced circuits $C_1$ and
$C_2$ joined by a path $P$, we switch the signature of $S$ to
make $P$ all-positive and construct a nowhere-zero $3$-flow as
follows. We send value $1$ along $C_1$ from the end-vertex of
$P$, value $-1$ along $C_2$ from the other end-vertex of $P$,
and value $-2$ along $P$ from the end-vertex in $C_1$ to the
end-vertex in $C_2$. Since Kirchhoff's law is satisfied at
every vertex of the bicircuit, the result is a nowhere-zero
$3$-flow.

Now we are in position to present a characterisation of
unbalanced flow-admissible graphs. Much of the result has been
previously known: the equivalence $(a)\Leftrightarrow (b)$
follows from the combination of Bouchet's Proposition 3.1 in
\cite{Bouchet} with Zaslavsky's characterisation of circuits in
the signed graphic matroid \cite[Theorem 5.1~(e)]{Zaslav2}; a
direct graph-theoretical proof can be found in
\cite[Corollary~3.2]{flowdecomp}. The implication
$(a)\Rightarrow (c)$ is a strengthening of Lemma~2.5 from
\cite{Bouchet}, and a special case of the equivalence
$(a)\Leftrightarrow (c)$ for antibalanced signed graphs has
been proved by Akbari et al. in \cite[Theorem~1]{AGKM} using a
different terminology. Nevertheless, to our best knowledge the
complete statement and the proof have never appeared in the
literature.

\begin{theorem}\label{thm:flow-admiss}
The following statements are equivalent for every connected
unbalanced signed graph~$G$.
\begin{itemize}
\item[{\rm (a)}] $G$ admits a nowhere-zero integer flow.

\item[{\rm (b)}] The edges of $G$ can be covered by signed
    circuits.

\item[{\rm (c)}] For each edge $e$, the graph $G-e$
    contains no balanced component.
\end{itemize}
\end{theorem}
\begin{proof}
We only prove that $(a)\Leftrightarrow (c)$. For the proof of
$(a)\Leftrightarrow (b)$ see \cite{Bouchet} or
\cite{flowdecomp}.

\medskip

$(a)\Rightarrow (c)$ Let $\xi$ be a nowhere-zero integer flow
on $G$, and suppose that, for some edge~$e$, the graph $G-e$
has a balanced component $H$. Switch the signature of $G$ to
make each edge of $H$ positive. If $e$ was a bridge in $G$,
then, as in the unsigned case, the sum of outflows from the
vertices of $H$ would force $\xi(e)=0$, which is impossible.
Therefore $G-e$ is connected and $e$ is the only negative edge
of $G$. However, by Lemma~\ref{lemma:extro}, the sum of values
on the negative edges taken with the extroverted orientation is
$0$, so $\xi(e)=0$, and we have a contradiction again.

\medskip

$(c)\Rightarrow (a)$ Assume that for every edge $e$ the graph
$G-e$ contains only unbalanced components. To see that $G$ is
flow-admissible we first show that each edge of $G$ belongs to
either a balanced circuit or to a weak unbalanced bicircuit. We
define a weak unbalanced bicircuit as a signed graph consisting
of two edge-disjoint unbalanced circuits $C_1$ and $C_2$, not
necessarily vertex-disjoint, joined with a path $P$, which may
be trivial, with no edge in $C_1\cup C_2$.

Let $e$ be an arbitrary edge of $G$, and suppose that it is not
contained in a balanced circuit. If $e$ is a bridge, then both
components of $G-e$ contain an unbalanced circuit. We take one
in each component and connect the circuits with a path, thus
creating a weak unbalanced bicircuit containing~$e$. If $e$ is
not a bridge, then it is contained in an unbalanced circuit,
say $C_1$. The graph $G-e$ is connected and unbalanced, so it
contains an unbalanced circuit $C_2$. If $C_1$ and $C_2$ are
edge-disjoint, then joining $C_1$ and $C_2$ with a path yields
a weak unbalanced bicircuit containing $e$. Therefore assume
that $C_1$ and $C_2$ have an edge in common, and consider the
modulo $2$ sum $C_1 \oplus C_2$ of $C_1$ and $C_2$. Since
$C_1\oplus C_2$ is an edge-disjoint union of circuits and $e$
is contained in $C_1\oplus C_2$, there is a circuit
$D_1\subseteq C_1\oplus C_2$ containing $e$. As there is no
balanced circuit through $e$, the circuit $D_1$ is unbalanced.
However, $\sigma(C_1\oplus C_2)=+1$, so $C_1\oplus C_2$ has to
contain an unbalanced circuit $D_2$ that is edge-disjoint from
$D_1$. By connecting $D_1$ and $D_2$ with a path we obtain a
weak unbalanced bicircuit containing~$e$. Thus every edge of
$G$ belongs to either a balanced circuit of a weak unbalanced
bicircuit.

Now let $\{B_1,B_2,\ldots,B_t\}$ be a covering of the edges of
$G$ such that each $B_i$ is either a balanced circuit or a weak
unbalanced bicircuit. If $B_i$ is a balanced circuit, then it
admits a nowhere-zero $2$-flow. If $B_i$ is a weak unbalanced
bicircuit, then it is easy to see that it has a nowhere-zero
$3$-flow. In both cases, there exists a nowhere-zero $3$-flow
$\phi_i$ on each~$B_i$. Regarding each $\phi_i$ as a flow on
the entire $G$ with zero values outside $B_i$ we can form the
function $\phi=\sum_{i=1}^t3^{i-1}\phi_i$ which is easily seen
to be a nowhere-zero $3^t$-flow on $G$.
\end{proof}

We say that a signed graph $G$ is \textit{tightly unbalanced} if
it is unbalanced and there is an edge $e$ such that $G-e$ is
balanced.

\begin{corollary}\label{cor:min-unbal}
A $2$-edge-connected unbalanced signed graph is flow-admissible if
and only if it is not tightly unbalanced.
\end{corollary}

The following immediate corollary establishes part~(a) of our
Main Theorem.

\begin{corollary}\label{cor:euler-flow-admiss}
A signed eulerian graph is flow-admissible if and only if it is
not tightly unbalanced.
\end{corollary}

We conclude this section with another result concerning
eulerian graphs that easily follows from general results
presented in this section. It is due to Xu and Zhang \cite{Xu}
and implies part (b) of Main Theorem. Here we provide a simple
proof.

Call a signed eulerian graph \textit{even} if it has an even
number of negative edges and \textit{odd} otherwise. Note that
the parity of the number of negative edges in an eulerian graph
is clearly preserved by every vertex switching. Hence, the
property of being even or odd is an invariant of the switching
class of an eulerian graph.

\begin{theorem}\label{thm:even} {\rm (Xu and Zhang \cite{Xu})}
A connected signed graph $G$ admits a nowhere-zero $2$-flow if
and only if $G$ is an even eulerian graph.
\end{theorem}

\begin{proof}
The condition is clearly sufficient, because sending the value $1$
along any eulerian trail produces a nowhere-zero $2$-flow. For the
converse, assume that $G$ is a connected graph that admits a
nowhere-zero $2$-flow $\xi$. At every vertex, the values of all
incident edges are either $+1$ or $-1$ and must sum to zero, so
the valency is even. By Lemma~\ref{lemma:extro}, the sum of values
on negative edges with extroverted orientation is $0$. Again, all
the summands are $+1$ or $-1$, so $G$ must have an even number of
negative edges.
\end{proof}

\section{Nowhere-zero 4-flows}

In this section we show that every flow-admissible signed
eulerian graph admits a nowhere-zero $4$-flow. By
Theorem~\ref{thm:even}, this is true for all even eulerian
graphs, so we can restrict here to odd eulerian graphs. In
fact, for odd eulerian graphs we prove a stronger result
providing several equivalent statements one of which is the
existence of a nowhere-zero $4$-flow.

Before proceeding to the result we need a simple lemma. Let $G$
and $H$ be two signed graphs with intersecting vertex sets and
assume that $G\cap H$ is equally signed in both $G$ and $H$.
Then $G\cup H$ will denote the signed graph where edges inherit
their signs from $G$ and $H$. If $G$ and $H$ are balanced
graphs, we say that $G\cap H$ is a \textit{consistent} subgraph
of $G$ and $H$ whenever $G\cap H$ has a balanced bipartition
that extends to a balanced bipartition of $G$ as well as to a
balanced bipartition of $H$.

\begin{lemma}\label{lemma:union}
Let $G$ and $H$ be balanced signed graphs. If $G\cap H$ is a
consistent subgraph of $G$ and $H$, then $G\cup H$ is balanced.
\end{lemma}

\begin{proof}
Let $G\cap H$ be a consistent subgraph of $G$ and $H$ and let
$U_1\cup U_2$ be a balanced bipartition of $G\cap H$ that can
be extended to a balanced bipartition $V_1\cup V_2$ of $G$ and
to a balanced bipartition $W_1\cup W_2$ of $H$. Without loss of
generality we may assume that $U_1\subseteq V_1$ and
$U_1\subseteq W_1$. Then every negative edge in $G\cup H$ is
between the sets $V_1\cup W_1$ and $V_2\cup W_2$ implying that
$G\cup H$ is a balanced graph.
\end{proof}

Here is the main result of this section.

\begin{theorem}\label{thm:4-flows}
The following statements are equivalent for every signed
eulerian graph $G$ with an odd number of negative edges.
\begin{itemize}
\item[{\rm (a)}] $G$ is flow-admissibble;
\item[{\rm (b)}] $G$ admits a nowhere-zero $4$-flow;
\item[{\rm (c)}] $G$ contains two edge-disjoint unbalanced
    circuits;
\item[{\rm (d)}] $G$ is a union of two even eulerian
    subgraphs;
\item[{\rm (e)}] $G$ can be decomposed into three
    edge-disjoint odd eulerian subgraphs.
\end{itemize}
\end{theorem}

\begin{proof}

$(a)\Rightarrow (c)$ Assume that $G$ is a flow-admissible odd
eulerian graph. We show that $G$ contains two edge-disjoint
unbalanced circuits.

Choose any unbalanced circuit $N$ of $G$, and set $G_0=G-E(N)$.
If $G_0$ is unbalanced, then it contains an unbalanced circuit,
which together with $N$ provides two required edge-disjoint
unbalanced circuits of $G$. Henceforth we can assume that $G_0$
is balanced.

Let us switch the signature of $G$ in such a way that $G_0$
becomes all-positive, and consider a fixed component $M$ of
$G_0$. The vertices of $M$ divide the circuit $N$ into
\textit{segments}, pairwise edge-disjoint paths whose first and
last vertex is in $M$ and all inner vertices lie outside~$M$. A
\emph{multisegment} is a portion of $N$ formed by a chain of
segments. Depending on the product of signs, segments can be
either positive or negative. Observe that if a segment $J$ is
negative, then $M\cup J$ is unbalanced, and vice versa.

\medskip\noindent
Claim 1. \textit{Every component of $G_0$ determines an odd
number of negative segments on $N$.}

\medskip\noindent
Proof of Claim 1. The product of signs of all segments
determined by a component of $G_0$ clearly equals the sign of
$N$. Since $N$ is unbalanced, the number of negative segments
must be odd. Claim~1 is proved.

\medskip

We now consider two cases.

\medskip

\textbf{Case 1.} \textit{$G_0$ is connected.} By Claim 1, $G_0$
determines at least one negative segment on~$N$. Let $S_1,
S_2,\ldots, S_k$ be all the segments of $N$, negative or not.
Suppose first that only one of them is negative, say $S_1$.
Then for each $S_i$ with $i\ge 2$ the graph $G_0\cup S_i$ is
balanced, and hence, by Lemma~\ref{lemma:union}, the graph
$G_0\cup S_2\cup S_3\ldots \cup S_k$ is also balanced. Since
$G_0\cup S_1$ is unbalanced, every unbalanced circuit in $G$
traverses $S_1$. But then for every edge $e$ of $S_1$ the graph
$G-e$ is balanced, contradicting Corollary~\ref{cor:min-unbal}.

Thus $G_0$ determines at least three negative segments on $N$.
We pick two of them, say $S_i$ and $S_j$. Clearly, there is an
unbalanced circuit $C_i$ in $G_0\cup S_i$ and an unbalanced
circuit $C_j$ in $G\cup S_j$. Nevertheless, $C_i$ and $C_j$ may
have a common edge within $G_0$. In order to handle this
problem, we extend $S_i$ and $S_j$ into negative multisegments
$S_i^+$ and $S_j^+$, respectively, with a vertex in common and
construct two unbalanced edge-disjoint circuits by utilising
$S_i^+$ and $S_j^+$ rather than $S_i$ and $S_j$.

Clearly, $S_i$ and $S_j$ have at most one vertex in common. If
$S_i$ and $S_j$ do have a vertex in common, we can set
$S_i^+=S_i$ and $S_j^+=S_j$. Otherwise we can express $N$ as
$S_iU_1S_jU_2$ for suitable multisegments $U_1$ and $U_2$. If
any of them, say $U_1$, is negative, then we set $S_i^+=S_i$
and $S_j^+=U_1$. So we may assume that both $U_1$ and $U_2$ are
positive, and then we may set $S_i^+=S_i$ and $S_j^+=U_1S_j$.
In each case we have found negative multisegments $S_i^+$ and
$S_j^+$ sharing precisely one common vertex.

Let $a$ and $b$ be the end-vertices of $S_i^+$ and let $b$ and
$c$ the end-vertices of~$S_j^+$. Since $G_0$ is eulerian, it
can be traversed by an eulerian trail $T$. The trail $T$
encounters each of vertices $a$, $b$, and $c$ at least once. It
follows that $T=T_1T_2T_3$ where $T_1$ is an
\mbox{$a$-$b$}-subtrail and $T_2$ is an $b$-$c$-subtrail of
$T$. Then $S_i^+T_1\inv $ and $S_j^+T_2\inv$ are two
edge-disjoint unbalanced closed trails. Each of them contains
an unbalanced circuit, so $G$ contains two edge-disjoint
unbalanced circuits.

\medskip

\textbf{Case 2.} \textit{$G_0$ is disconnected.} If $G_0$ has a
component $M$ that produces at least three negative segments,
we proceed as in Case 1. We may therefore assume that each
component of $G_0$ determines exactly one negative segment on
$N$.

If $G_0$ has two components $M_1$ and $M_2$ that determine
disjoint negative segments $S_1$ and $S_2$, respectively, then
each of $M_1\cup S_1$ and $M_2\cup S_2$ contains an unbalanced
circuit, and we are done. Consequently, we may assume that the
negative segments coming from any two components of $G_0$
intersect. Clearly, their intersection will consist of either
one or two paths. Since distinct components are disjoint, these
paths must be nontrivial. Let $M_1, M_2, \ldots, M_n$ be the
components of $G_0$, let $S_j$ be the negative segment of $N$
determined by $M_j$, and for $j=1,2,\ldots, n$ let $S_j'$ be
the complementary path on $N$ with the same end-vertices as
$S_j$. We call $S_j'$ the \textit{cosegment} of $M_j$.

\medskip\noindent
Claim 2. \textit{The cosegments cover all of $N$.}

\medskip\noindent
Proof of Claim 2. Suppose, to the contrary, that there exists
an edge $e$ of $N$ that does not belong to any cosegment. Then
$e$ belongs to every segment $S_j$ for $i=1,2,\ldots, k$. Let
$Q_j=(M_j\cup S_j)-e$ and let $R_j=Q_1\cup Q_2\cup\ldots\cup
Q_j$. Clearly, each $Q_j$ is balanced. By induction on $j$ we
next show that each $R_j$ is balanced, and using this fact we
derive a contradiction. Since $R_1=Q_1$, the basis of induction
is verified. Assume inductively that $R_j$ is balanced for some
$j$ with $1\le j\le n-1$. To prove that $R_{j+1}$ is balanced
we apply Lemma~\ref{lemma:union} to the graphs $R_j$ and
$Q_{j+1}$. Obviously, $Q_{j+1}\cap R_j$ is equally signed in
both $Q_{j+1}$ and~$R_j$. Furthermore, $Q_{j+1}\cap R_j$ is
contained in $N-e$ and consists of two paths, each having one
end-vertex of the edge $e$. Since $R_j$ and $Q_{j+1}$ are
balanced but $R_j+e$ and $Q_{j+1}+e$ are not, the ends of $e$
in both $R_j$ and $Q_{j+1}$ either belong to the same partite
set or to different partite sets. This immediately implies that
$Q_{j+1}\cap R_j$ is a consistent subgraph of $Q_{j+1}$ and
$R_j$, and by Lemma~\ref{lemma:union}, $Q_{j+1}\cup
R_j=R_{j+1}$ is a balanced graph. Thus $R_j$ is balanced for
each $j=1,2,\ldots, n$. This means, however, that $G-e=R_n$ is
balanced, contradicting Corollary~\ref{cor:min-unbal}. The
proof of Claim~2 is complete.

\medskip\noindent
Claim 3. \textit{Let $\mathcal{S}'$ be a minimal covering of
$N$ by cosegments. Then}
\begin{itemize}
\item[(i)] \textit{every component of the intersection of
    any two cosegments is a nontrivial path;}

\item[(ii)] \textit{each edge of $N$ is covered by either
    one cosegment or by two cosegments.}
\end{itemize}

Proof of Claim 3. Part (i) is trivial. To prove (ii) suppose,
to the contrary, that there exists an edge of $N$ that belongs
to three cosegments $J'$, $K'$, and $L'$ from $\mathcal{S}'$.
Then $J'\cup K'\cup L'$ is either a path or the whole of $N$.
In the former case, the end-vertices are in at most two of
$J'$, $K'$, and $L'$. The remaining cosegment is therefore
contained in the union of the other two, implying the covering
$\mathcal{S}'$ is not minimal. If $J'\cup K'\cup L'=N$ and,
say, $J'\cup K'\ne N$, then $L'$ must have a disconnected
intersection with one of $J'$ and $K'$, say~$J'$. But then
$K'\subseteq J'\cup L'$, again contradicting the minimality
of~$\mathcal{S}'$. This proves Claim~3.

\medskip

Fix a cyclic ordering of the vertices of $N$ and let
$\mathcal{S}'=\{ S_1', S_2',\ldots, S_m'\}$ where
$S_i'=N[u_i,v_i]$ is the portion of $N$ with end-vertices $u_i$
and $v_i$ following this cyclic ordering for each
$i\in\{1,2,\ldots,m\}$. By Claim~2, we can assume that the
members of $\mathcal{S}'$ are arranged in such a way that each
$S_i'$ only intersects its predecessor $S_{i-1}'$ and its
successor $S_{i+1}'$, and the vertices $u_i$ and $v_i$ occur on
$N$ in the cyclic ordering
$u_1,v_m,u_2,v_1,\ldots,u_m,v_{m-1},u_1$.

We now construct two edge-disjoint unbalanced circuits in $G$.
Consider an arbitrary component $M_i$ with
$i\in\{1,2,\ldots,m\}$. We can arrange the edges of each $M_i$
into two edge-disjoint $u_i$-$v_i$-trails $X_i$ and $Y_i$ such
that $X_iY_i\inv$ is an eulerian trail of~$M_i$. Since $M_i\cup
S_i'$ is balanced, we have
\begin{eqnarray}
\sigma(X_i)=\sigma(Y_i)=\sigma(S_i')=\sigma(N[u_i,v_i]).\label{eq:1}
\end{eqnarray}
If $m$ is even, the following two closed trails $T_1$ and $T_2$
in $G$ are clearly edge-disjoint:
\begin{align*}
T_1&=X_1N[v_1,u_3]X_3 N[v_3, u_5]\ldots X_{m-1}
N[v_{m-1},u_1],\\
T_2&=X_2N[v_2,u_4]X_4N[v_4,u_6]\ldots X_{m}
N[v_{m},u_2].
\end{align*}
From (\ref{eq:1}) we get
\begin{align*}
\sigma(T_1)&=\sigma(N[u_1,v_1])\sigma(N[v_1,u_3])\sigma(N[u_3,v_3])
             \ldots\sigma(N[u_{m-1},v_{m-1}])\sigma(N[v_{m-1},u_1])\\
           &=\sigma(N)=-1,
\end{align*}
and similarly $\sigma(T_2)=-1$. Hence both $T_1$ and $T_2$ are
unbalanced, and therefore each of them contains an unbalanced
circuit. This provides the two required edge-disjoint
unbalanced circuits in $G$.

If $m$ is odd, we have the following two edge-disjoint closed
trails $T_1$ and $T_2$ in $G$:
\begin{align*}
T_1&=X_1N\inv[v_1,u_2]X_2N[v_2,u_4]X_4\ldots
X_{m-1}N[v_{m-1},u_1],\\
T_2&=Y_2N\inv[v_2,u_3]X_3N[v_3,u_5]X_5\ldots
X_{m}N[v_{m},u_2].
\end{align*}
From (\ref{eq:1}) we obtain
\begin{align*}
\sigma(X_1)&=\sigma(N[u_1,v_1])=\sigma(N[u_1,u_2])\sigma(N[u_2,v_1]),\\
\sigma(X_2)&=\sigma(N[u_2,v_2])=\sigma(N[u_2,v_1])\sigma(N[v_1,v_2]).
\end{align*}
Hence
\begin{eqnarray}
\sigma(X_1N_1\inv[v_1,u_2]X_2)=\sigma(N[u_1,u_2])\sigma(N[u_2,v_1])
\sigma(N[v_1,v_2])=\sigma(N[u_1,v_2]),\label{eq:2}
\end{eqnarray}
and analogously,
\begin{eqnarray}
\sigma(Y_2N\inv[v_2,u_3]X_3)=\sigma(N[u_2,v_3]).\label{eq:3}
\end{eqnarray}
Employing equations (\ref{eq:1})-(\ref{eq:3}) in a similar
fashion as for $m$ even we derive that
$\sigma(T_1)=\sigma(T_2)=-1$. Therefore each of $T_1$ and $T_2$
contains an unbalanced circuit, providing two edge-disjoint
unbalanced circuits in $G$, as required. This completes the
proof of $(a)\Rightarrow (c)$.

\medskip

$(c)\Rightarrow (e)$ Let $G$ be an odd eulerian graph
containing two edge-disjoint unbalanced circuits. We show that
it can be decomposed into three edge-disjoint odd eulerian
subgraphs.

Clearly, $G$ admits a circuit decomposition $\mathcal{K}$ that
contains the two unbalanced circuits as its members. Since $G$
is odd, the decomposition $\mathcal{K}$ will have an odd number
of unbalanced circuits, and therefore at least three unbalanced
circuits. Let us consider the incidence graph $J(\mathcal{K})$
of~$\mathcal{K}$; its vertices are the elements of
$\mathcal{K}$ and edges join pairs of elements that have a
vertex of $G$ in common. Since $G$ is connected, so is
$J(\mathcal{K})$.

It is obvious that every connected induced subgraph of
$J(\mathcal{K})$, with vertex set a subset
$\mathcal{L}\subseteq\mathcal{K}$, uniquely determines an
eulerian subgraph of $G$. The latter subgraph will have an odd
number of negative edges whenever $\mathcal{L}$ contains an odd
number of unbalanced circuits. Thus to finish the proof it is
enough to show that $\mathcal{K}$ can be partitioned into three
subsets, each containing an odd number of unbalanced circuits
and each inducing a connected subgraph of~$J(\mathcal{K})$. In
fact, we may assume that $J(\mathcal{K})$ is a tree as the
general case follows immediately with the partition of
$\mathcal{K}$ obtained from a spanning tree of
$J(\mathcal{K})$.

Thus, let $J(\mathcal{K})=T$ be a tree. We may think of the
vertices of $T$ as being coloured in two colours:
\textit{black}, if the corresponding circuit in $\mathcal{K}$
is unbalanced, and \textit{white}, if the corresponding circuit
is balanced. In this terminology, it remains to prove the
following.

\medskip\noindent
Claim~4. \textit{Let $T$ be a tree whose vertices are
partitioned into two subsets, white vertices and black
vertices, such that the number of black vertices is odd and at
least $3$. Then the vertex set of $T$ can be partitioned into
three subsets such that each contains an odd number of black
vertices and induces a subtree of $T$.}

\medskip\noindent
Proof of Claim~4. We proceed by induction on the number of
vertices of $T$. The conclusion is obvious if $T$ has only
three vertices, each of them black. This constitutes the basis
of induction. For the induction step we assume that $T$ has
four or more vertices. It follows that at least two of them,
say $v_1$ and $v_2$, are leaves of $T$. If both $v_1$ and $v_2$
are black, then the set $\{\{v_1\},\{v_2\},V(T)-\{v_1,v_2\}\}$
is the required decomposition. If one of them is white,
say~$v_1$, then by the induction hypothesis $T-v_1$ has the
required decomposition $\{V_1, V_2, V_3\}$. One of these sets,
say $V_1$, contains a neighbour of $v_1$, and then
$\{V_1\cup\{v_1\}, V_2, V_3\}$ is the required decomposition
for $T$. This concludes the induction step and as well as the
proof the implication $(c)\Rightarrow (e)$.

\medskip

$(e)\Rightarrow (d)$ Assume that $G$ has a decomposition
$\{G_1,G_2,G_3\}$ into three odd eulerian subgraphs. Without
loss of generality we may assume that $G_1$ and $G_2$ share a
vertex $u$ and that $G_2$ and $G_3$ share a vertex $v$;
possibly $u=v$. Then $G_1\cup G_2$ and $G_2\cup G_3$ are even
eulerian subgraphs that cover $G$.

\medskip

$(d)\Rightarrow (b)$ Assume that $G$ is an odd eulerian graph
that has a covering $\{H_1,H_2\}$ by two even eulerian
subgraphs. By Theorem~\ref{thm:even}, there exists a
nowhere-zero $2$-flow $\phi_1$ on $H_1$ and a nowhere-zero
$2$-flow $\phi_2$ on $H_2$. Regarding each $\phi_i$ as a flow
on the entire $G$ with zero values outside $H_i$, we can set
$\phi=\phi_1+2\phi_2$. It is obvious that $\phi$ is a
nowhere-zero $4$-flow on $G$.

\medskip

$(b)\Rightarrow (a)$ This implication is trivial.
\end{proof}

Theorem~\ref{thm:4-flows} and
Corollary~\ref{cor:euler-flow-admiss} have the following
interesting consequence.

\begin{corollary}
Let $G$ be an signed eulerian graph with an odd number of
negative edges. If any two unbalanced circuits of $G$ have an
edge in common, then there exists an edge $e$ which is
contained in all unbalanced circuits. In particular, $G$ is
tightly unbalanced.
\end{corollary}

\section{Nowhere-zero 3-flows}

The aim of this section is to establish part (c) of our Main
Theorem. We have to show that a signed eulerian graph $G$ has
flow number three if and only if it can be decomposed into
three odd eulerian subgraphs $G_1$, $G_2$, and $G_3$ that have
a vertex in common. If $G$ has such a decomposition, we say
that it is \textit{triply odd}. The decomposition
$\{G_1,G_2,G_3\}$ will itself be called \textit{triply odd}. 
Obviously, a triply odd signed eulerian graph is odd.

It is immediate that a signed eulerian graph is triply odd if
and only if its edges can be arranged into three unbalanced
closed trails originating at the same vertex. The following
fact is a direct consequence of this observation.

\begin{proposition}\label{prop:3-sufficiency}
Let $G$ be a triply odd signed eulerian graph. Then $\Phi(G)=3$.
\end{proposition}

\begin{proof}
Let $\{G_1, G_2, G_3\}$ be a triply odd decomposition of $G$
where $G_1$, $G_2$, and $G_3$ share a vertex $v$. For
$i\in\{1,2,3\}$ let $T_i$ be an eulerian trail in $G_i$
starting at $v$. If we send from $v$ the value $1$ along $T_1$
and $T_2$, and the value $-2$ along $T_3$, the resulting
valuation will clearly be a nowhere-zero $3$-flow on $G$. Since
$G$ is odd, Theorem~\ref{thm:even} implies that
$\Phi(G)\ge 3$. Hence $\Phi(G)=3$.
\end{proof}

The rest of this section is devoted to proving the reverse
implication. The proof has two main ingredients: a reduction of
the general case to antibalanced $6$-regular graphs and a
verification that the result holds for connected antibalanced
$6$-regular graphs with an odd number of edges (or
equivalently, with an odd number of vertices). The latter fact
is nontrivial and can be derived from the following theorem
which is the main result of~\cite{6-reg}.

\begin{theorem}~\label{thm:6regular}
Every connected $6$-regular graph of odd order can be
decomposed into three eulerian subgraphs sharing a vertex such
that each of them has an odd number of edges.
\end{theorem}

As regards the reduction procedure, we begin with analysing 
signed eulerian graphs that carry a special type of a $3$-flow
under which the same edge-value does not enter and
simultaneously leave any vertex. The formal definition uses an
orientation where each edge is assigned the direction with
positive flow value. We call this orientation a
\textit{positive orientation of} $G$ \textit{with respect to} a
given nowhere-zero flow. Now let $\phi$ be a nowhere-zero
$3$-flow on a signed eulerian graph $G$, and let $G$ be
positively oriented with respect to $\phi$. We say that $\phi$
is \textit{stable at a vertex $v$} if for any two edges $e$ and
$f$ incident with $v$ such that $\phi(e)=\phi(f)$, either both
$e$ and $f$ are directed towards $v$ or both are directed out
of~$v$. A nowhere-zero $3$-flow $\phi$ is said to be
\textit{stable} if it is stable at every vertex.

\begin{lemma}~\label{lemma:stableflow}
Let $G$ be a signed eulerian graph that admits a stable
nowhere-zero $3$-flow. Then $G$ is antibalanced and the valency of
every vertex is a multiple of $6$.
\end{lemma}

\begin{proof}
Let $\phi$ be a stable nowhere-zero $3$-flow on $G$ which is
positively oriented. Consider an arbitrary vertex $v$ of $G$.
Clearly, all the edges from $E(v)$ that carry the same value
$a\in\{1,2\}$ under $\phi$ have the same orientation with
respect to $v$ -- either to or from~$v$. Thus the situation at
$v$ is that either all edges from $E(v)$ with value $1$ are
directed to $v$ and those with value $2$ are directed out
of~$v$ (\textit{Type} 1), or vice versa (\textit{Type} 2).

By the Kirchhoff law, there must be an integer $m$ such that
$E(v)$ has $m$ edges with value $2$ and $2m$ edges with value~$1$.
Hence, the valency of $v$ equals $3m$, but since $3m$ is an even
integer, we infer that the valency of $v$ is $6n$ for some $n$.

Finally, observe that an edge joining two vertices of the same
type is negative whereas an edge joining two vertices of
different types is positive. By
Corollary~\ref{col:antibal-bipart}, the partition of the vertex
set of $G$ into the vertices of Type 1 and those of Type 2 is
an antibalanced bipartition. Hence, $G$ is antibalanced, as
claimed.
\end{proof}

We are now in the position to prove the charactarisation of
signed eulerian graphs with flow number three.

\begin{theorem}\label{thm:3-flows}
Let $G$ be a connected signed eulerian graph. Then $\Phi(G)=3$ if
and only if $G$ is triply odd.
\end{theorem}

\begin{proof}
In Proposition~\ref{prop:3-sufficiency} we have proved that the
condition is sufficient. It remains to prove its necessity. Let
$G$ be a signed eulerian graph with $\Phi(G)=3$. By
Theorem~\ref{thm:even}, $G$ is unbalanced and odd. To show that
$G$ is triply odd we proceed by induction on the cycle rank
$\beta(G)$ of~$G$. Recall that the cycle rank of a graph is the
dimension its cycle space, and that a graph $G$ with $k$
components has $\beta(G)=|E(G)|-|V(G)|+k$.

If $G$ contains a $2$-valent vertex incident with edges $e_1$
and $e_2$, we may suppress the vertex and form a new edge $e$
whose sign equals the product of signs of $e_1$ and $e_2$. The
result is a signed eulerian graph $G'$ with same cycle rank and
the same flow number. This allows us to assume, whenever
convenient, that the valency of each vertex of $G$ is at least
$4$. With this additional assumption we obtain that
$\beta(G)\ge |V(G)|+1$, further implying that up to a
homeomorphism the only signed eulerian graph $G$ with
$\Phi(G)=3$ and $\beta(G)\le 3$ is the bouquet of three
unbalanced loops. Its cycle rank is $3$, and for this graph the
result clearly holds. This verifies the basis of induction.

For the induction step let $G$ be a signed eulerian graph with
$\Phi(G)=3$ and cycle rank $\beta(G)>3$, and assume that the
assertion holds for all signed eulerian graphs with cycle rank
smaller than the cycle rank of $G$. Suppose to the contrary
that for $G$ the assertion fails. Then $G$ is a minimum
counterexample, and our aim is to derive a contradiction
for~$G$.

\medskip\noindent
Claim 1. \textit{A nowhere-zero $3$-flow on a minimum
counterexample is stable.}

\medskip\noindent
Proof of Claim 1. Let $\phi$ be a nowhere-zero $3$-flow on $G$.
Suppose that there exists a vertex $v$ at which $\phi$ is not
stable, and let $G$ be positively directed with respect to
$\phi$. Then $E(v)$ contains a pair of edges $e$ and $f$ with
$\phi(e)=\phi(f)$ such that $e$ is directed to $v$ and $f$ is
directed out of~$v$.

If $e$ and $f$ coincide, then $e$ must be a positive loop. It
follows that $G-e$ is an odd eulerian graph carrying a
nowhere-zero $3$-flow, so $\Phi(G-e)=3$,  by
Theorem~\ref{thm:even}. Since $\beta(G-e)<\beta(G)$, the graph
$G-e$ has a triply odd decomposition $\{G_1, G_2, G_3\}$. By
adding $e$ to any $G_i$ we obtain a triply odd decomposition
for $G$, a contradiction.

Now suppose that $e$ and $f$ are distinct edges. In this case
we remove $e$ and $f$ from $G$ and replace the path $ef$ with a
single edge~$g$ whose sign equals the product of signs of $e$
and~$f$. In the resulting graph $G'$, all vertices have an even
valency and the number of negative edges has the same parity as before.
Due to the consistent orientation of the inner half-edges of
the path~$ef$, the new edge $g$ has a natural bidirection
determined by the outer half-edges of $ef$ and this bidirection
is consistent with the sign of $g$. Furthermore, setting
$\phi(g)=\phi(e)=\phi(f)$ turns $\phi$ into a nowhere zero
$3$-flow on $G'$.

If $G'$ is connected, then it is an odd eulerian graph,
implying that $\Phi(G')=3$. As $\beta(G')< \beta(G)$, we can
find a triply odd decomposition $\{G_1', G_2', G_3'\}$ of $G'$
with common vertex $v$. One of the subgraphs contains the
edge~$g$; we may assume that this graph is~$G_1$. Then
$G_1=(G_1'-g)\cup\{e,f\}$ is an odd eulerian subgraph of $G$,
and consequently $\{G_1,G_2',G_3'\}$ is a triply odd
decomposition of $G$ with common vertex $v$, a contradiction.

If $G'$ is disconnected, then it has exactly two components $H$
and $K$, both eulerian, one of which, say $H$, is odd. Since
$H$ carries a nowhere-zero $3$-flow, we have $\Phi(H)=3$.
However, $\beta(H)<\beta(G)$, so $H$ has a triply odd
decomposition $\{H_1,H_2,H_3\}$. Reinserting the vertex $v$
into the edge $g$ restores the edges $e$ and $f$ without
changing the parity of the number negative edges in the
component of $G'$ containing $g$. It is therefore possible
convert the decomposition $\{H_1,H_2,H_3\}$ of $H$ into a
triply odd decomposition of $G$, a contradiction again.

In all possible cases the assumption that the $3$-flow $\phi$
is unstable produces a contradiction. It follows that $\phi$ is
stable, as claimed.

\medskip\noindent
Claim 2. \textit{A minimum counterexample is an antibalanced
$6$-regular graph.}

\medskip\noindent
Proof of Claim 2. Again, let $G$ be a minimum counterexample.
As we have just shown, $G$ carries a stable nowhere-zero
$3$-flow. By Lemma~\ref{lemma:stableflow}, $G$ is antibalanced
and the valency of every vertex is a positive multiple of $6$.
It remains to prove $G$ is $6$-regular.

Suppose that $G$ contains a vertex $v$ of valency $6n$ for some
$n>1$. Replace $v$ with two new vertices $v'$ and $v''$ and
join the edges originally incident with $v$ to the two new
vertices in such a way that $v'$ becomes $6$-valent, $v''$
becomes $6(n-1)$-valent, and the resulting graph $G'$ continues
to carry a stable nowhere-zero $3$-flow. This is clearly
possible. Since $G'$ is odd, we have $\Phi(G')=3$.

If $G'$ is connected, then $\beta(G')<\beta(G)$, and by the
induction hypothesis $G'$ has a triply odd decomposition. This
decomposition readily induces one for $G$, a contradiction. If
$G'$ is disconnected, then it has two components $G_1$ and
$G_2$, only one of which, say~$G_1$, is odd. Since
$\Phi(G_1)=3$ and $\beta(G_1)<\beta(G)$, there exists a triply
odd decomposition in~$G_1$. By extending one of the subgraphs
of the decomposition with $G_2$ we obtain a similar
decomposition for $G$, a contradiction again.

All this shows that $G$ cannot have a vertex of valency $6n$
for $n>1$. Therefore $G$ is $6$-regular, and the proof of Claim
2 is complete.

\medskip

Now we can finish the induction step. As above, let $G$ be a
minimum counterexample with $\Phi(G)=3$. From Claim~2 we know
that $G$ is an antibalanced $6$-regular graph. Since $G$ is
odd, it has an odd number of edges, and hence an odd order. In
this situation we can apply Theorem~\ref{thm:6regular} to infer
that $G$ is triply odd contradicting the assumption that $G$ is
a counterexample. This concludes the induction step as well as
the proof of the theorem.
\end{proof}

\section{Group-valued flows}

We conclude this paper with a brief discussion of nowhere-zero
flows with values in abelian groups other than the group of
integers. Our final result provides a complete characterisation
of signed eulerian graphs that admit a nowhere-zero $A$-flow
for a given abelian group $A\ne 0$.

\begin{theorem}
Let $G$ be a signed eulerian graph and let $A$ be a nontrivial
abelian group. The following statements hold true.
\begin{itemize}
\item[{\rm (a)}]{If $A$ contains an involution, then $G$
    admits a nowhere-zero $A$-flow.}
\item[{\rm (b)}]{If $A\cong \mathbb{Z}_3$, then $G$ admits
    a nowhere-zero $A$-flow if and only if $G$ is triply
    odd.}
\item[{\rm (c)}]{Otherwise, $G$ has a nowhere-zero $A$-flow
    if and only if $G$ is not tightly unbalanced.}
\end{itemize}
\end{theorem}

\begin{proof}
(a) If $A$ contains an involution, then a nowhere-zero $A$-flow
can be produced by simply valuating each edge of $G$ by the
same involution.

(b) Assume that $A\cong \mathbb{Z}_3$. By \cite[Theorem
1.5]{Xu}, a 2-edge-connected signed graph admits a nowhere-zero
$\mathbb{Z}_3$-flow if and only if it admits a nowhere-zero
3-flow. Our Theorem~\ref{thm:3-flows} now implies that $G$ has
a nowhere-zero $A$-flow if and only if it is triply odd.

(c) Assume that $A$ contains no involution and is not
isomorphic to $\mathbb{Z}_3$. Then either $A$ contains a
subgroup $B$ isomorphic to $\mathbb{Z}_3\times\mathbb{Z}_3$ or
has an element of order at least $4$. If $G$ is an even
eulerian graph, then sending any nontrivial element of $A$
along an eulerian trail of $G$ will produce a nowhere-zero
$A$-flow on $G$. Let $G$ be odd but not tightly unbalanced.
By Theorem~\ref{thm:4-flows}~(d), we can cover $G$ with two
even eulerian subgraphs, say $T_1$ and $T_2$. If $A$ contains
$B\cong \mathbb{Z}_3\times\mathbb{Z}_3$, we take any generating
set $\{b_1,b_2\}$ for $B$ and send the value $b_1$ along an
eulerian trail in $T_1$ and the value $b_2$ along an eulerian
trail in $T_2$. If $A$ has an element of order at least 4, say
$a$, we proceed similarly with $a$ in $T_1$ and $2a$ in $T_2$.
In both cases we get a nowhere-zero $A$-flow on $G$. Finally,
let $G$ be tightly unbalanced. It is easy to see that the
argument from the proof of Lemma~\ref{lemma:extro} extends to
any group with no involution, and hence applies to $A$. It
follows that $G$ has no nowhere-zero $A$-flow, and the proof is
complete.
\end{proof}

\vskip0.5cm

\textbf{Acknowledgements} 

We are grateful to Andr\'e Raspaud and Xuding Zhu for a fruitful 
discussion on the topic of our present paper. 

Research reported in this paper was supported by the following
grants. The first author was partially supported by VEGA 1/0876/16 
and the second author was partially supported by VEGA
1/0474/15. Both authors were also supported from APVV-0223-10
and APVV-15-0220.

\end{document}